\documentclass[11pt,reqno]{amsart}
\usepackage{amssymb,amsmath,amsthm,newlfont,enumerate}

\DeclareMathOperator{\rank}{rank}

\theoremstyle{plain}
\newtheorem{theorem}{Theorem}[section]

\newtheorem{lemma}[theorem]{Lemma}

\theoremstyle{definition}

\theoremstyle{remark}

\newcommand{\cF}{\mathcal{F}}

\newcommand{\cL}{\mathcal{L}}
\newcommand{\CC}{\mathbb{C}}
\newcommand{\DD}{\mathbb{D}}
\newcommand{\RR}{\mathbb{R}}

\newcommand{\TT}{\mathbb{T}}

\DeclareMathOperator{\supp}{supp}
\DeclareMathOperator{\dist}{dist}

\begin{document}

\title[Negative powers of Hilbert-space contractions]{A proof of  Esterle's conjecture  on negative powers of Hilbert-space contractions}

\author[T. Ransford]{Thomas Ransford}
\address{D\'epartement de math\'ematiques et de statistique, Universit\'e Laval, Qu\'ebec (QC), G1V 0A6, Canada}
\email{ransford@mat.ulaval.ca}

\date{20 Aug 2026}

\begin{abstract}
We establish the following result,  confirming a conjecture of Jean Esterle.
For each closed subset $E$ of the unit circle  of Lebesgue measure zero, there exists a
positive  sequence $u_n\to\infty$
with the following property: 
if $T$ is a contraction on a Hilbert space such that $\sigma(T)\subset E$ and 
$\|T^{-n}\|=O(u_n)$ as $n\to\infty$, then $T$ is a unitary operator.  

A key tool used in the proof is a result
generalizing the well-known fact that closed subsets $E$ of the real axis of Lebesgue measure
zero are removable for bounded holomorphic functions. We show that such sets remain
removable even for  certain
unbounded holomorphic functions of moderate growth near $E$,
where the notion of `moderate' depends on $E$.
\end{abstract}

\thanks{Research supported  by NSERC Discovery Grant RGPIN-2026-04565}

\keywords{Contraction, unitary, inner function, removable singularity}

\makeatletter
\@namedef{subjclassname@2020}{\textup{2020} Mathematics Subject Classification}
\makeatother

\subjclass[2020]{30J05, 47A45, 47A56}

\maketitle

\section{Introduction and statement of main results}\label{S:Intro}

Let $T$ be a bounded linear operator on a complex Hilbert space.
We denote the operator norm and spectrum of $T$ by $\|T\|$ and $\sigma(T)$ respectively. We say $T$ is a \emph{contraction} if $\|T\|\le 1$.

Our goal in this article is to establish the following theorem,
which confirms a conjecture  of Jean Esterle \cite[Conjecture~1.3]{Ra24}.  

\begin{theorem}\label{T:Esterle}
For each closed subset $E$ of the unit circle $\TT$ of Lebesgue measure zero, there exists a
positive  sequence $u_n\to\infty$
with the following property: 
if $T$ is a contraction on a Hilbert space such that $\sigma(T)\subset E$ and 
$\|T^{-n}\|=O(u_n)$ as $n\to\infty$, then $T$ is a unitary operator.
\end{theorem}

Versions of this result have long been known for specific sets $E$: 
for example countable compact sets \cite{Za93},
certain Cantor sets \cite{Es94b}, and Carleson sets \cite{Es94a,Ke98}. 
In each case, the corresponding sequence $(u_n)$ was explicitly identified. It is also known that no single sequence $(u_n)$
can work simultaneously for all compact sets $E$ of measure zero
\cite{Ke98}. Also the condition that $E$ have measure zero is optimal: Theorem~\ref{T:Esterle} breaks down for every closed set $E$ of positive measure \cite{Ra24}. Finally, the result also breaks down if `Hilbert space'
is replaced by `Banach space' \cite{Es94a}. For a detailed account 
of the background to Esterle's conjecture, we refer
to \cite{Ra24}.

In  \cite[Theorem~1.5]{Ra24} it was shown that 
Theorem~\ref{T:Esterle} holds under the additional assumption
that $\rank(I-T^*T)<\infty$. In this note, we prove the full version, following a method suggested to the author by Fedja Nazarov. 
The new ingredient is a
removable singularity theorem, which we believe to be of interest in its own right.
It is a generalization of the well-known fact that a closed 
subset $E$ of the real line $\RR$ of Lebesgue measure zero
is removable for bounded holomorphic functions.
The new result shows that this remains true even for certain
unbounded holomorphic functions of moderate growth near $E$,
where what `moderate' means depends on $E$.

We denote by $|\cdot|$  Lebesgue measure on $\RR$.
Also, given $E\subset\CC$ and $t>0$, 
we write $E_t:=\{z\in\CC:\dist(z,E)\le t\}$.

\begin{theorem}\label{T:removsing}
Let $E$ be a compact subset of $\RR$ such that $|E|=0$.
\begin{enumerate}[\normalfont(i)]
\item 
There exists a decreasing continuous function
$\omega:(0,\infty)\to(1,\infty)$ such that
\begin{equation}\label{E:removsing}
\lim_{t\to0^+}\omega(t)=\infty
\quad\text{and}\quad
\liminf_{t\to0^+}\frac{|E_t\cap\RR|}{t}\int_0^t\omega(s)\,ds=0.
\end{equation}
\item
Every $\omega$ satisfying the conditions in (i)
has the following  property:
if $\Omega$ is an open subset of $\CC$ containing $E$, and
if $f$ is a holomorphic function on $\Omega\setminus E$ 
such that $|f(x+iy)| \le \omega(|y|)$
for all $x+iy\in\Omega\setminus\RR$,
then $f$ has a holomorphic extension to the whole of $\Omega$.
\end{enumerate}
\end{theorem}


\section{Proof of Theorem~\ref{T:removsing}}\label{S:removsing}

The proof uses two lemmas. In the first of these, $m$ denotes planar Lebesgue measure.

\begin{lemma}\label{L:removsing}
Let $\Omega$ be an open subset of $\CC$,
and let $F$ be a compact subset of $\CC$ such that $m(F)=0$.
Then every holomorphic function $f:\Omega\setminus F\to\CC$ 
such that
\begin{equation}\label{E:cond}
\liminf_{\eta\to0}\frac{1}{\eta}\int_{F_\eta}|f(z)|\,dm(z)=0
\end{equation}
has a holomorphic extension to the whole of $\Omega$.
\end{lemma}

The idea for the proof that follows was suggested  by Fedja Nazarov.

\begin{proof}
Fix $\psi\in C^\infty(\CC)$ with $\supp\psi\subset\DD$ and $\int_\CC\psi\,dm=1$.
For each $\eta>0$, define $\phi_\eta:\CC\to\RR$  by
\[
\phi_\eta(z):=1-\frac{1}{(\eta/4)^2}\int_{F_{\eta/2}}\psi\Bigl(\frac{z-w}{\eta/4}\Bigr)\,dm(w).
\]
Then $\phi_\eta\in C^\infty(\CC)$  and it satisfies 
$\phi_\eta=1$ on $\CC\setminus F_\eta$ and 
$\phi_\eta=0$ on $F_{\eta/4}$. 
Also
 $|\partial\phi_\eta/\partial\overline{z}|\le C_\psi/\eta$ 
on~$\CC$, where $C_\psi$ is a constant
depending just on $\psi$.

Let $f:\Omega\setminus F\to\CC$ be a holomorphic function
such that \eqref{E:cond} holds.
Define $f_\eta:\Omega\to\CC$ by
\[
f_\eta(z):=
\begin{cases}
f(z)\phi_\eta(z), &z\in\Omega\setminus F,\\
0 , &z\in F_{\eta/4}.
\end{cases}
\]
Then
$f_\eta\in C^\infty(\Omega)$,
and it satisfies
$f_\eta=f$ on $\Omega\setminus F_\eta$
and $\partial f_\eta/\partial\overline{z}=f\partial\phi_\eta/\partial\overline{z}$ on $\Omega\setminus F$.
Fix $\Gamma$, a finite union of closed contours in $\Omega\setminus F$ such that
\[
n(\Gamma,w)=
\begin{cases}
0 \text{~or~} 1, &\forall w\in\CC\setminus\Gamma,\\
1, &\forall w\in F,\\
0, &\forall w\in\CC\setminus\Omega.
\end{cases}
\]
Set $U:=\{w\in\Omega\setminus\Gamma: n(\Gamma,w)=1\}$, an open neighbourhood of $F$.
By the Cauchy--Pompeiu formula,
\[
f_\eta(\zeta)=\frac{1}{2\pi i}\int_\Gamma \frac{f_\eta(z)}{z-\zeta}\,dz
-\frac{1}{\pi}\int_{U}\frac{\partial f_\eta}{\partial\overline{z}}(z)\frac{1}{z-\zeta}\,dm(z) \quad(\zeta\in U).
\]
In particular, if $\eta<\dist(\Gamma,F)$,
then
\[
f(\zeta)=\frac{1}{2\pi i}\int_\Gamma \frac{f(z)}{z-\zeta}\,dz
-\frac{1}{\pi}\int_{F_\eta}f(z)\frac{\partial \phi_\eta}{\partial\overline{z}}(z)\frac{1}{z-\zeta}\,dm(z) \quad(\zeta\in U\setminus F_\eta).
\]
For each $\zeta\in U\setminus F_\eta$, we have
\[
\Bigl|\frac{1}{\pi}\int_{F_\eta}f(z)\frac{\partial \phi_\eta}{\partial\overline{z}}(z)\frac{1}{z-\zeta}\,dm(z)\Bigr|
\le \frac{C_\psi}{\pi\dist(\zeta,F_\eta)}\frac{1}{\eta}\int_{F_\eta}|f(z)|\,dm(z).
\]
Taking $\liminf_{\eta\to0}$ and using \eqref{E:cond}, we deduce that
\[
f(\zeta)=\frac{1}{2\pi i}\int_\Gamma \frac{f(z)}{z-\zeta}\,dz
\quad(\zeta\in U\setminus F).
\]
As a function of $\zeta$,
the right-hand is holomorphic on $U$, so it provides the required
holomorphic extension of $f$.
\end{proof}

\begin{lemma}\label{L:realvar}
Let $\rho:(0,\infty)\to(0,\infty)$ be a continuous increasing function such that $\lim_{t\to 0^+}\rho(t)/t=\infty$.
Then there exists a continuous decreasing function $\omega:(0,\infty)\to(1,\infty)$ such that
\[
\lim_{t\to0^+}\omega(t)=\infty
\quad\text{and}\quad
\liminf_{t\to0^+}\frac{1}{\rho(t)}\int_0^t\omega(s)\,ds=0.
\]
\end{lemma}

\begin{proof}
Define $\rho_1:(0,\infty)\to(0,\infty)$ by $\rho_1(t):=\sqrt{t\rho(t)}$.
Then 
\[
\lim_{t\to0+}\frac{\rho_1(t)}{t}=\lim_{t\to0^+}\sqrt{\frac{\rho(t)}{t}}=\infty,
\]
so we may recursively choose a decreasing sequence $(t_n)$ tending to zero such that, if we set
\[
d_n:=\frac{\rho_1(t_{n-1})-\rho_1(t_{n})}{t_{n-1}-t_{n}},
\]
then $(d_n)$ is an increasing sequence that tends to infinity as $n\to\infty$.
Fix $n_0$ such that $d_{n_0}>1$, and define $\omega:(0,\infty)\to(1,\infty)$ by stipulating 
that $\omega(t_{n})=d_{n}$ for all $n\ge n_0$, that $\omega$ is linear on each interval $[t_{n+1},t_n]$
for $n\ge n_0$ and constant on $[t_{n_0},\infty)$. Clearly $\omega$ is a continuous decreasing
function such that $\lim_{t\to0^+}\omega(t)=\infty$. Also, for each $n\ge n_0$, we have
\begin{align*}
\int _0^{t_n}\omega(s)\,ds
&\le \sum_{k\ge n}\int_{t_{k+1}}^{t_k}\omega(s)\,ds
\le \sum_{k\ge n}\omega(t_{k+1})(t_k-t_{k+1})\\
&=\sum_{k\ge n}(\rho_1(t_k)-\rho_1(t_{k+1}))
\le \rho_1(t_n),
\end{align*}
and hence
\[
\frac{1}{\rho(t_n)}\int_0^{t_n}\omega(s)\,ds
=\frac{t_n}{\rho_1(t_n)^2}\int_0^{t_n}\omega(s)\,ds
\le \frac{t_n}{\rho_1(t_n)}.
\]
Since $t_n/\rho_1(t_n)\to0$, it follows that
\[
\lim_{n\to\infty}\frac{1}{\rho(t_n)}\int_0^{t_n}\omega(s)\,ds=0,
\]
which establishes the result.
\end{proof}

\begin{proof}[Proof of Theorem~\ref{T:removsing}]
(i) 
Define $\rho:(0,\infty)\to(0,\infty)$ by
\[
\rho(t):=\frac{t}{|E_t\cap \RR|} \quad(t>0).
\]
Then $\rho$ is a continuous increasing function
such that $\lim_{t\to0^+}\rho(t)/t=\infty$.
By Lemma~\ref{L:realvar}, there exists a
continuous decreasing function $\omega:(0,\infty)\to(1,\infty)$
such that \eqref{E:removsing} holds.

(ii) Let $\omega$ be a function satisfying the conclusions of (i),
let $\Omega$ be an open subset of $\CC$ containing $E$,
and let $f$ be a holomorphic function on $\Omega\setminus E$
such that $|f(x+iy)|\le \omega(|y|)$ for all $x+iy\in\Omega\setminus\RR$.
For each $\eta>0$, we have 
\[
E_\eta\subset (E_\eta\cap\RR)\times [-\eta,\eta],
\]
so,
if $\eta$ is small enough so that $(E_\eta\cap\RR)\times [-\eta,\eta]\subset\Omega$, then
\[
\int_{E_\eta}|f|\,dm\le \int_{E_\eta\cap\RR}\int_{-\eta}^\eta |f(x+iy)|\,dy\,dx
\le 2|E_\eta\cap\RR|\int_{0}^\eta \omega(y)\,dy.
\]
Hence
\[
\liminf_{\eta\to0^+}\frac{1}{\eta}\int_{E_\eta}|f|\,dm\le \liminf_{\eta\to0^+}\frac{2|E_\eta\cap\RR|}{\eta}\int_0^\eta\omega(y)\,dy=0,
\]
the last  equality by \eqref{E:removsing}.
By Lemma~\ref{L:removsing}, the function $f$ has a holomorphic extension to the whole of $\Omega$.
This shows that $\omega$ has the required property, and completes the proof of the theorem.
\end{proof}


\section{Proof of Theorem~\ref{T:Esterle}}\label{S:Esterle}

Given Hilbert spaces $\cF,\cF'$, we write $\cL(\cF,\cF')$ for the space of 
bounded linear operators from $\cF$ into $\cF'$.
Also we say that $T\in\cL(\cF,\cF')$ is \emph{purely contractive} if $\|Tx\|<\|x\|$ for all $x\in\cF\setminus\{0\}$.

The following result is a key step in the proof of
Theorem~\ref{T:Esterle}.

\begin{theorem}\label{T:inner}
For each closed subset $E$ of $\TT$ 
of Lebesgue measure zero,
 there exists a positive sequence $u_n\to\infty$
with the following property:
if $\cF,\cF'$ are  Hilbert spaces
and if $\Theta:(\CC_\infty\setminus E)\to\cL(\cF,\cF')$ is a 
non-constant holomorphic function such that
 $\Theta(\lambda)$ is  contractive for all $\lambda\in\DD$ and  unitary
 for all $\lambda\in\TT\setminus E$,
then
\[
\liminf_{n\to\infty}u_n\delta_n(\Theta)=0,
\]
where
\begin{equation}\label{E:deltandef}
\delta_n(\Theta):=\inf_{\lambda\in\DD}\max\bigl\{|\lambda|^n,\|\Theta(\lambda)^{-1}\|^{-1}\bigr\}.
\end{equation}
\end{theorem}

In \cite[Theorem~5.7]{Ra24} it was shown that 
Theorem~\ref{T:inner} holds in the special case
where $\dim\cF=\dim\cF'<\infty$ (leading to the version of Theorem~\ref{T:Esterle}
where $\rank(I-T^*T)<\infty$).
The technique used was to reduce to the scalar case
$\dim_\cF=\dim\cF'=1$ by taking determinants, and then to 
exploit the fact that, in that case,
$\Theta$ is a singular inner function, so has a representation in terms of a singular measure
on $E$. In the case where $\dim \cF=\dim\cF'=\infty$, we know of no such representation,
and so the problem of proving Theorem~\ref{T:inner} was left open \cite[Conjecture~7.1]{Ra24}.
Using the removable singularity result Theorem~\ref{T:removsing}, we can now solve this problem.

Before embarking upon on the proof of Theorem~\ref{T:inner},
it will be convenient to establish an alternative formula
for $\delta_n(\Theta)$.

\begin{lemma}\label{L:deltan}
Let $E,\cF,\cF',\Theta,\delta_n(\Theta)$ be as in the
statement of Theorem~\ref{T:inner}. Then
\begin{equation}\label{E:deltan}
1/\delta_n(\Theta)=\sup_{|\lambda|>1}\min\bigl\{|\lambda|^n,\|\Theta(\lambda)\|\bigr\}.
\end{equation}
\end{lemma}

\begin{proof}
The function $\lambda\mapsto\Theta(\lambda)\Theta(1/\overline{\lambda})^*$ is holomorphic on 
$\CC_\infty\setminus E$ and equal to the identity $I$ 
when $\lambda\in\TT\setminus E$.
By the identity principle, it follows that 
$\Theta(\lambda)\Theta(1/\overline{\lambda})^*=I$
for all $\lambda\in\CC_\infty\setminus E$.
The formula \eqref{E:deltan} follows easily.
\end{proof}

\begin{proof}[Proof of Theorem~\ref{T:inner}]
Let $E$ be a closed subset of $\TT$ of Lebesgue measure zero.
By Theorem~\ref{T:removsing} and a simple conformal-mapping argument,
there exists a decreasing continuous function $\omega:(0,\infty)\to(1,\infty)$
satisfying $\lim_{t\to0^+}\omega(t)=\infty$ with the following property:
if $\Omega$ is an open subset of $\CC$ containing $E$,
and if $f:\Omega\setminus E\to\CC$ is a holomorphic function
such that $|f(re^{i\theta})|\le \omega(|r-1|)$ for all $re^{i\theta}\in\Omega\setminus\TT$,
then $f$ has a holomorphic extension to the whole of $\Omega$.

For  $n\ge1$, the function $t\mapsto\omega(t)-(1+t)^n:(0,\infty)\to\RR$ is continuous, strictly decreasing,
tends to $+\infty$ as $t\to0^+$ and to $-\infty$ as $t\to\infty$.
Therefore there exists a unique $t_n\in (0,\infty)$ such that $\omega(t_n)=(1+t_n)^n$.
Clearly the sequence $(t_n)$ is strictly decreasing and $t_n\to0$ as $n\to\infty$.
Set 
\[
u_n:=(1+t_{n+1})^{n/2} \quad(n\ge1).
\]
Note that
\[
u_n=\omega(t_{n+1})^{n/(2n+2)} \quad(n\ge1),
\]
which implies that $u_n\to\infty$.
We shall show that this sequence $(u_n)$ has the property prescribed in the statement of the theorem.

Let $\cF,\cF'$ be Hilbert spaces, 
and let $\Theta:(\CC_\infty\setminus E)\to\cL(\cF,\cF')$ be a 
non-constant holomorphic function such that
$\Theta(\lambda)$ is contractive for all $\lambda\in\DD$ and  unitary
for all $\lambda\in\TT\setminus E$. Since $\Theta$ is non-constant,
there exist unit vectors $v\in\cF$ and $v'\in\cF'$ such that,
if we define the holomorphic function $f:\CC_\infty\setminus E\to\CC$ by
 \[
 f(\lambda):=\langle \Theta(\lambda)v,v'\rangle \quad(\lambda\in\CC_\infty\setminus E),
 \]
then $f$ is non-constant.
This implies that the set $E$ is not a removable singularity for $f$.
By the choice of $\omega$ above, it follows that there exists a sequence $(\lambda_k)$ in $\CC\setminus \TT$
such that $|\lambda_k|\to1$ and $|f(\lambda_k)|>\omega(||\lambda_k|-1|)$ for all $k$.
As $|f|\le 1$ on $\DD$, all but finitely many $(\lambda_k)$ lie outside $\overline{\DD}$,
so we may as well suppose that $|\lambda_k|>1$ for all $k$. For each $k$ sufficiently large, 
there exists an $n\ge1$ such that $1+t_{n+1}\le |\lambda_k|\le 1+t_n$. Then, by the maximum principle
applied on the exterior of $|\lambda|=1+t_{n+1}$,
\[
|f(\lambda_k)|\le \max_{|\lambda|=1+t_{n+1}}|f(\lambda)|\le\max_{|\lambda|=1+t_{n+1}}\|\Theta(\lambda)\|,
\]
and hence
\[
1< \frac{|f(\lambda_k)|}{\omega(|\lambda_k|-1)}
\le \max_{|\lambda|=1+t_{n+1}}\frac{\|\Theta(\lambda)\|}{\omega(|\lambda_k|-1)}
\le \max_{|\lambda|=1+t_{n+1}}\frac{\|\Theta(\lambda)\|}{\omega(t_n)}.
\]
By our choice of $t_n$, we have $\omega(t_n)=(1+t_n)^n$, and so
\[
1< \max_{|\lambda|=1+t_{n+1}}\frac{\|\Theta(\lambda)\|}{(1+t_n)^n}
\le \max_{|\lambda|=1+t_{n+1}}\frac{\|\Theta(\lambda)\|}{(1+t_{n+1})^n}
=\max_{|\lambda|=1+t_{n+1}}\frac{\|\Theta(\lambda)\|}{|\lambda|^n}.
\]
Using \eqref{E:deltan}, we deduce that
\[
\frac{1}{\delta_n(\Theta)}=\sup_{|\lambda|>1}\min\bigl\{|\lambda|^n,\|\Theta(\lambda)\|\bigr\}
\ge (1+t_{n+1})^n=u_n^2,
\]
the last equality by our choice of $u_n$. To summarize, we have shown that,
for each $n$ such that $1+t_{n+1}\le |\lambda_k|\le 1+t_n$ for some $k$,
we have $\delta_n(\Theta)\le 1/u_n^2$. Since $|\lambda_k|\to1$, there are
infinitely many such $n$. Consequently
\[
\liminf_{n\to\infty}u_n\delta_n(\Theta)\le \liminf_{n\to\infty}1/u_n=0.
\]
Thus the sequence $(u_n)$ has the required property.
\end{proof}

It remains to show how Theorem~\ref{T:Esterle} can be
deduced using Theorem~\ref{T:inner}.
In fact this part of the argument was already explained in detail in \cite{Ra24},
so we shall just sketch the proof for the sake of completeness.

\begin{proof}[Proof of Theorem~\ref{T:Esterle}]
Let $E$ be a closed subset of $\TT$ of Lebesgue measure zero.
Let $(u_n)$ be a positive sequence tending to infinity and satisfying
the conclusion of Theorem~\ref{T:inner}. 
We shall show that,
if $T$ is a non-unitary Hilbert-space contraction such that
$\sigma(T)\subset E$, then $\limsup_{n\to\infty}\|T^{-n}\|/u_n
=\infty$, thereby proving Theorem~\ref{T:Esterle}.

Let $T$ be a non-unitary Hilbert-space contraction such that
$\sigma(T)\subset E$. By \cite[Lemma~6.1]{Ra24},
there is a completely non-unitary contraction $T_1$
defined on a separable Hilbert space such that
$\sigma(T_1)\subset\sigma(T)$ and $\|T_1^{-n}\|\le\|T^{-n}\|$
for all $n$. Thus, we may as well suppose from the outset
that $T$ is completely non-unitary and that it is defined
on a separable Hilbert space. Together with the fact that
$\sigma(T)$ is a subset $\TT$ of Lebesgue measure zero, 
this implies that both $T^n$ and $T^{*n}$ converge strongly to zero as $n\to\infty$ \cite[Lemma~6.2]{Ra24}.

We now invoke a version of the Nagy--Foias model theorem for contractions
\cite[Theorem~5.8]{Ra24}: since $T$ is a  contraction on a separable
Hilbert space such that $T^{*n}$ converges strongly to zero as $n\to\infty$, it is unitarily equivalent to $S_\Theta$, 
the compressed shift associated to a
non-constant operator-valued 
inner function $\Theta$ on~$\DD$. Thus $\|T^{-n}\|=\|S_\Theta^{-n}\|$
for all $n\ge1$, and from \cite[Theorem~5.5]{Ra24} we have
\[
\|S_\Theta^{-n}\|\ge \frac{1}{2}\Bigl(\frac{1}{\delta_n(\Theta)}-1\Bigr)
\quad(n\ge1),
\]
where $\delta_n(S_\Theta)$ is defined as in \eqref{E:deltandef}.

By \cite[Theorem~5.2]{Ra24}, since $\sigma(T)\subset E$,
it follows that $\Theta$ may be chosen so that it has a holomorphic
extension to the whole of $\CC\setminus E$ and so that $\Theta(\lambda)$ is
unitary for all $\lambda\in\TT\setminus E$.
Thus $\Theta$ satisfies all the hypotheses of
Theorem~\ref{T:inner}, and so by that theorem we deduce that
\[
\liminf_{n\to\infty}u_n\delta_n(\Theta)=0.
\]
Putting everything together, we  obtain
\[
\limsup_{n\to\infty}\|T^{-n}\|/u_n
=\limsup_{n\to\infty}\|S_\Theta^{-n}\|/u_n
\ge(1/2)(\liminf_{n\to\infty}u_n\delta_n(\Theta))^{-1}
=\infty,
\]
as was to be shown.
\end{proof}


\section*{Acknowledgements}
I am grateful to Fedja Nazarov for a valuable discussion on this topic,
and in particular for suggesting the idea used in the proof of Lemma~\ref{L:removsing}.

While this paper was being prepared, I learned from William Verreault that he too had
(independently) found a proof of Theorem~\ref{T:inner}. His proof, which can be found in \cite{Ve26},
is based on notions from potential theory,
and is different from the one presented here.
I thank him for sharing his ideas with me.

\bibliographystyle{plain}
\bibliography{biblist}

\end{document}